\documentclass[11pt]{amsart}
\usepackage{amsmath,amsthm,amssymb,amsfonts,amscd}
\usepackage{amsxtra}
\usepackage{eucal}
\usepackage{mathrsfs}
\usepackage{color, graphics}
\usepackage[all]{xy}
\usepackage{hyperref}
\usepackage{pstricks,pst-plot}
\usepackage{cases}
\headsep=1cm \numberwithin{equation}{section}
\SelectTips{cm}{11}

\newcommand \Bl {\ensuremath{\mathrm{Bl}}}

\newcommand \ini {\ensuremath{\mathrm{in}}}
\newcommand \Tor {\ensuremath{\mathrm{Tor}}}
\newcommand \Trisec {\ensuremath{\mathrm{Trisec}}}

\newcommand \reg {\mathrm {reg}}
\newcommand \Sec {\ensuremath{\mathrm{Sec}}}

\newcommand \Sym {\ensuremath{\mathrm{Sym}}}

\newcommand \codim {\ensuremath{\mathrm{codim}}}

\newcommand \depth {\ensuremath{\mathrm{depth}}}
\newcommand \pd {\ensuremath{\mathrm{pd}}}
\def\P{{\mathbb P}}
\def\Z{{\mathbb Z}}
\def\N{{\mathbb N}}

\newcommand \st[1] {\stackrel{#1}{\longrightarrow}}
\newcommand \sts[1] {\stackrel{#1}{\rightarrow}}
\newcommand \iso {\stackrel{\sim}{\rightarrow}}
\newcommand \isost[1] {\stackrel{#1}{\iso}}
\newcommand \sst[1] {\stackrel{#1}{\twoheadrightarrow}}
\newcommand \ist[1] {\stackrel{#1}{\hookrightarrow}}

\newcommand \lra {\rightarrow}
\newcommand \llra {\longrightarrow}
\newcommand \p {\partial}
\newcommand\keyword[1]{\par{\bf {\it Keywords:~}}{#1}}

\newcommand \rrlap[1]{\hbox to 0pt{#1}}

\theoremstyle{theorem} 
\newtheorem{Thm}{Theorem}[section]
\newtheorem{Prop}[Thm]{Proposition}
\newtheorem{Coro}[Thm]{Corollary}

\theoremstyle{definition}
\newtheorem{Def}[Thm]{Definition}
\newtheorem{Ex}[Thm]{Example}
\newtheorem{Qu}[Thm]{Question}

\newtheorem{Remk}[Thm]{Remark}

\newtheorem{Obs}[Thm]{Observation}

\numberwithin{equation}{section}

\begin{document}

\title{Analysis on some infinite modules, inner projection, and applications}
\author[K.\ Han and S.\ Kwak]{Kangjin Han and Sijong Kwak}
\address{Algebraic Structure and its Applications Research Center(ASARC), Department of Mathematics, Korea Advanced Institute of Science and Technology,
373-1 Gusung-dong, Yusung-Gu, Daejeon, Korea}
\email{han.kangjin@kaist.ac.kr}
\address{Department of Mathematics, Korea Advanced Institute of Science and Technology,
373-1 Gusung-dong, Yusung-Gu, Daejeon, Korea}
\email{skwak@kaist.ac.kr}
\thanks{
The authors were supported by Basic Science Research Program through the National Research Foundation
of Korea(NRF) funded by the Ministry of Education, Science and Technology(grant No.2009-0063180)}


\maketitle
\begin{center}
\textit{Dedicated to the memory of Hyo Chul Myung\\
(June 15, 1937 - February 11, 2010)}
\end{center}
\bigskip

\begin{abstract}

A projective scheme $X$ is called `quadratic' if $X$ is scheme-theoretically cut out by homogeneous equations of degree $2$. Furthermore, we say $X$ satisfies `property $\textbf{N}_{2,p}$' if it is quadratic and the quadratic ideal has only linear syzygies up to first $p$-th steps. In the present paper, we compare the linear syzygies of the inner projections with those of $X$ and obtain a theorem on `embedded linear syzygies' as one of our main results. This is the natural projection-analogue of `restricting linear syzygies'
in the linear section case, \cite{EGHP1}. As an immediate corollary, we show that the inner projections of $X$ satisfy property $\textbf{N}_{2,p-1}$
for any reduced scheme $X$ with property $\textbf{N}_{2,p}$.

Moreover, we also obtain the neccessary lower bound $(\codim X)\cdot p -\frac{p(p-1)}{2}$
, which is sharp, on the number of quadrics vanishing on $X$ in order to satisfy $\textbf{N}_{2,p}$ and show that the arithmetic depths of inner projections are equal to that of the quadratic scheme $X$. These results admit an interesting `syzygetic' rigidity theorem on property $\textbf{N}_{2,p}$ which leads the classifications of extremal and next to extremal cases.

For these results we develope the elimination mapping cone theorem for infinitely generated
graded modules and improve the partial elimination ideal theory initiated by M. Green. This new method allows us to treat a wider class of 
projective schemes which can not be covered by the Koszul cohomology techniques, because these are not projectively normal in general.
 
\bigskip
\noindent\keyword{linear syzygies, the mapping cone theorem, partial elimination ideals, inner projection, arithmetic depth, Castelnuovo-Mumford regularity.}
\end{abstract}

\tableofcontents \setcounter{page}{1}

\section*{Introduction}\label{section_1}
Let $X$ be a nondegenerate reduced closed subscheme in a projective space $\mathbb P^N$ over an algebraically
closed field $k$ of characteristic zero and $R=k[x_0,\ldots,x_N]$ be the coordinate ring of $\P^N$. The equations defining $X$ and the syzygies among them have played a central role to study projective schemes in algebraic geometry. Further the syzygy structures and their geometric implications have been intensively focused for the most interesting case, i.e. projective schemes having property 
$\textbf{N}_{2,p}$ for last twenty years, see \cite{CKK,EGHP1,EGHP2,EHU,GL}.
They are closely related to the Eisenbud-Goto conjecture on Castelnuovo-Mumford regularity and other conjectures on linear syzygies
in classical algebraic geometry. The linear sections and projections of $X$ have been very useful to understand those problems.

For the linear sections, we have interesting results on `Restricting linear syzygies' due to Eisenbud, Green, Hulek, Popescu, see \cite{EGHP1}. Along this line, a natural question could be raised:

\begin{center}
\textit{What is the relations between the syzygies of projections and $X$?}\\
\end{center}

In the present paper, we especially consider the relations between the linear syzygies of \textit{inner} projections and those of $X$. Note that the inner projection has been a standard issue classically since del Pezzo and Fano used this projection for the classification of del Pezzo surfaces and Fano 3-folds, see \cite{R}. There are also some known results about non-birational loci of these projection morphisms and geometric structures of the projection images, see 
\cite{BHSS,CC,Se,So}.\\

{\textbf{Problems}} We list our main problems in detail:
\begin{enumerate}
\item [(a)](Embedded linear syzygies) Let $X$ be a nondegenerate reduced scheme in $\P^N$
satisfying property $\textbf {N}_{2,p}, (p\ge 1)$. Consider the linear projection from a linear subvariety $\Lambda\subset \P^N$ of $\dim \Lambda=t < p$ with $\Lambda\cap X\neq \emptyset$, $\langle \Lambda\cap X\rangle =\Lambda$ and $X_\Lambda=\overline{\pi_{\Lambda}(X\setminus \Lambda)}$
in $\P^{N-t-1}$. How do the syzygies behave under projections? D. Eisenbud et. al. showed that under some $\textbf{N}_{2,p}$-assumption, the syzygies of $X$
restrict surjectively to the syzygies of linear sections in their paper `Restricting linear syzygies', \cite{EGHP1}. Is there any natural projection-analogue of the linear section case? Bearing on this problem, we also expect that $X_\Lambda$ satisfies property $\textbf {N}_{2, p-t-1}$.
  
\item [(b)](Necessary lower bound for property $\textbf {N}_{2, p}$) For a quadratic scheme $X$ satisfying $\textbf{N}_{2,p}$, it is roughly believed that the more quadratic equations $X$ has, the further steps linear syzygies proceed to. Therefore one can ask `how many quadrics does $X$ require to satisfy property $\textbf{N}_{2,p}$?' This is a natural question, but not yet known.

\item [(c)](Syzygetic rigidity theorem) In \cite{EGHP1,EGHP2} they also show that a closed subscheme $X\subset \P^N$ is $2$-regular if $X$ satisfies property $\textbf {N}_{2, \codim X}$ and characterize all $2$-regular algebraic sets geometrically for this extremal case. What about `next to extremal case', i.e. a scheme $X$ satisfying $\textbf {N}_{2, \codim X -1}$? How to classify or characterize them in a suitable category?
\end{enumerate}

For those problems, we develope the elimination mapping cone theorem for infinitely generated
graded modules and improve the partial elimination ideal theory initiated by M. Green for the inner projection. This allows us to treat a wider class of 
projective schemes which can not be covered by the Koszul cohomology techniques, because these are not projectively normal in general. We have also found it very interesting to understand some relations between the syzygies of its projections and those of $X$ as we move the center of projection.\\

{\bf Organization of the paper} We recall basic definitions and preliminaries in Section \ref{section_1.5}. In Section \ref{section_2}, we set up the elimination mapping
cone construction for infinitely generated graded modules and the partial elimination ideal theory for inner projection which are crucial to understand
the syzygy structures of infinitely generated graded modules. This partial elimination ideal theory gives us \textit{local}
information of $X$ near the center of projection $q\in X$ which turns out to govern syzygies and other properties of the inner projection $X_q$ from the (\textit{global})
homogeneous equations.

In Section \ref{section_3}, we obtain some results on syzygy structures and geometric properties of {\it inner} projections,
i.e. embedded linear syzygies, the number of quadratic equations, and their corollaries. In particular, we can show that for any projective reduced scheme $X$ satisfying property $\textbf{N}_{2,p}$ the inner projection from any smooth point satisfies at least property $\textbf{N}_{2,p-1}$ and $X_\Lambda$ satisfies at least $\textbf{N}_{2,p-t-1}$ for a \textit{general} $t$-dimensional linear subspace $\Lambda$ with $\dim X\cap\Lambda=0$ (see Corollary \ref{Main result N_{2,p}} and Remark \ref{remk_for_N2,p}). We also give the neccessary lower bound $(\codim X)\cdot p -\frac{p(p-1)}{2}$ on the number of quadrics vanishing on $X$ in order to satisfy property $\textbf{N}_{2,p}$, which is sharp.

In Section \ref{section_3.5} we prove that the arithmetic depths of inner projections are equal to that of the given quadratic scheme. Combined with results in the previous section, this $depth$ theorem leads us to a very interesting `syzygetic' rigidity theorem on property $\textbf{N}_{2,p}$ in the category of varieties, namely, for the extremal (i.e. $p=\codim X$) and next to extremal (i.e. $p=\codim X -1$) cases those varieties should be \textit{arithmetically Cohen-Macaulay (abbr. ACM)} and we can give the classfications of the two cases. We also extend this result to more general category (See Corollary \ref{rigidity_2} and Question \ref{Q_classify}). Finally, in Section \ref{section_5} we see some examples and open questions stimulating further work.\\

{\bf Acknowledgements} The first author would like to thank Professor Frank-Olaf Schreyer for hosting his visit to Saarbr\"ucken under KOSEF-DAAD Summer Institute Program and for many valuable comments preparing this paper. The second author would like to thank Professor B. Sturmfels, M. Brodmann for their useful comments, especially P. Schenzel for valuable discussion and Example ~\ref{schenzel} during their stay in Korea Institute of Advanced Study(KIAS) and KAIST, Korea in the Summer 2009. We would also like to thank Professor F. Zak who informed us of Professors A. Alzati and J.C. Sierra's recent paper \cite{AS} related to our paper (see Remark \ref{Al-Si}).

\section{Definitions and Preliminaries}\label{section_1.5}

We work over an algebraically
closed field $k$ of characteristic zero. Let $X$ be a nondegenerate reduced closed subscheme in a projective space $\mathbb P^N$.

\begin{Def}\label{basic_def}
Let $X$ be as above.
\begin{itemize}
\item[(a)] $X$ is said to be a \textit{quadratic scheme} if there is a homogeneous ideal $I$ generated by equations of degree $2$ which defines $X$ scheme-theoretically (i.e. its sheafification $\widetilde{I}$ is equal to the ideal sheaf $\mathcal I_X$ of $X$).
\item[(b)] $X$ is said to \textit{satisfy property} $\textbf{N}_{2,p}$ \textit{scheme-theoretically} if it is quadratic and the quadratic ideal $I$ has only linear syzygies at least up to first $p$-th steps.
\item[(c)] $X$ is said to be \textit{$m$-regular} if $H^i(\mathcal I_X (m-i))=0 \text{ for all }i\geq 1$. We call $\reg(X):=\min\{\,m\,|\, H^i(\mathcal I_X (m-i))=0 \text{ for all }i\geq 1\}$ \textit{Castelnuovo-Mumford regularity of $X$}.
\end{itemize}
\end{Def}

Note that Definition \ref{basic_def} (b) is a generalization of known notions. It is the same as property $\textbf{N}_p$ defined by Green-Lazarsfeld if $X$ is projectively normal and $I$ is saturated (see \cite{GL}) .\\

Let $\pi_{\Lambda}:X\subset \P^N\dashrightarrow \P^{N-t-1}$ denote the projection of $X$ from a linear space $\Lambda=\P^{t}$.
We call it either {\it outer} projection if $X\cap \Lambda=\emptyset$ or {\it inner} projection in case $\Lambda\subset X$.
Every projection $\pi_{\Lambda}$ can be regarded as succesive compositions of suitable outer and inner projections from points.
These projections as well as blow-ups have been very useful projective techniques in algebraic geometry. We briefly review the preliminaries about an inner projection from a point $q\in X$.\\

Let $\sigma : \widetilde{X}\to X$ be a blowing up of $X$ at a smooth point $q\in X$. One has the regular morphism
$\pi^{'}:\widetilde{X}\twoheadrightarrow{X_q}:=\overline{\pi_q(X\setminus \{q\})}\subset \P^{N-1}$ with the following diagram;

\begin{equation*}\label{inner-diagram}
    \xymatrix @!=2pc{
 &&&\P^N\times\P^{N-1}\supset\widetilde X\phantom{XXXXXXX}\ar[dr]^{\pi'}
 \ar[dl]_{\sigma}\\
  &&**[l]\P^N\supset X\ar@{-->}[rr]_{\pi_q }&&**[r]X_q
  =\overline{\pi_q(X\setminus q)}\subset \P^{N-1}}
\end{equation*}

Classically, one says that a smooth variety $X$ {\it{admits an inner projection}} if $\pi^{'}$ is an embedding for some point $q\in X$. This is equivalent to $q\in X\setminus{\Trisec}(X)$
where $\Trisec(X)$ is the union of all lines $\ell$ with the condition that $\ell\subset X$ or $X\cap\ell$ is a subscheme of length at least $3$. We also know that the exceptional divisor $E$ is linearly
embedded via $\pi^{'}$ in $\P^{N-1}$ (i.e. $\pi^{'}(E)= \P^{r-1}\subset\P^{N-1}$, $r=\dim X$) for any subvariety $X$ if the center $q$ is smooth (see \cite{B,FCV}).\\

Let $R=k[x_0,\ldots,x_N]$ and $S=k[x_1,x_{2}\ldots,x_N]$ be the homogeneous coordinate rings of $\P^N$ and $\P^{N-1}$. Assume $q=(1:0:\ldots:0)\in X$ (by suitable coordinate change). Let $I$ be an ideal of $R$ defining a reduced scheme $X$ scheme-theoretically. Naturally we can give a scheme structure on the image $X_q$ by the ideal $J:=I\bigcap S$. Note that the ideal $J$ is reduced if $I$ is reduced.\\

In case of inner projection we note that $R/I$ is not finitely generated $S$-module, because $q\in X$ and there is no polynomial of the form
$f={x_0}^{n} + (\text {\,other \,\,terms\,})$ for some $n\in\N$ in the ideal $I$, even though $R/I$ is finitely generated as $R$-module. $I$ is also an \textit{infinitely generated} graded $S$-module with the following resolution :
\[
\begin{array}{cccccccccccccccccccc}
&&&&\cdots \lra & \bigoplus_{j=2}^{\infty} S(-i-j)^{\beta_{i,j}}&\lra\cdots\lra&\bigoplus_{j=2}^{\infty} S(-j)^{\beta_{0,j}}&\lra&I&\lra&0 ~.
\end{array}
\]

In Section \ref{section_2}, we show that they have interesting syzygy structures as $S$-modules (see Proposition \ref{linear} and Remark \ref{commute}).\\

On the other hand, if $X$ is quadratic, then we can write the quadratic ideal $I$ as
\begin{equation}\label{write_I}
~~~ I=(x_0 \ell_1-Q_{0,1},\ldots,x_0 \ell_t-Q_{0,t},Q_1,\ldots,Q_s), ~~q=(1,0,\ldots,0)\in X
\end{equation}
where $\ell_i$ is a linear form, $Q_{0,i},Q_j$ are quadratic forms in
$S=k[x_1,\ldots,x_N]$ and they are minimal generators. We can also assume all $\{\ell_i\}$ are linearly independent, and all $\{Q_{0,i}\}$ are
distinct. Clearly, $\{\ell_i\}$ generate $(T_{q}X)^{*}$. Note that $t=\codim(X)=N-\dim X$ if $q$ is a smooth point.
In general, $t$ is equal to $N-\dim T_{q}X$.\\

{\bf Convention} We are working on the following convention:
\begin{itemize}
\item Let $X\subset\P^N$ and $q\in X$ be as above and $I$ be a homogeneous defining ideal of $X$. We denote the $S$-ideal $I\cap S$ by $J$ which gives the natural induced scheme structure on the projection image $X_q\subset\P^{N-1}$ and call $J$ \textit{the $x_0$-elimination ideal of $I$}. In addition, we write the saturated ideal defining $X$ (resp. $X_q$) as $I_X$ (resp. $I_{X_q}$).
\item (Betti numbers) For a graded $R$-module $M$, we define \textit{graded Betti numbers $\beta^R_{i,j}(M)$ of $M$} by $\dim_k\Tor_i^{R}(M,k)_{i+j}$. We consider  $\beta^S_{i,j}(N)$ for any graded $S$-module $N$ in the same manner. We remind readers that $\Tor^R_i(R/I,k)_{i+j}=\Tor^R_{i-1}(I,k)_{i-1+j+1}$. So $\beta^R_{i,j}(R/I)=\beta^R_{i-1,j+1}(I)$.
We write $\beta^R_{i,j}(M)$ as $\beta_{i,j}(M)$ or $\beta_{i,j}$ if it is obvious.
\item We often abbreviate $\Tor_i^{R}(M,k)_{i+j}$ as $\Tor_i^{R}(M)_{i+j}$ (same  for $S$-module $\Tor$).
\item (Arithmetic depth) When we refer the \textit{depth of $X$}, denoted by $\depth_R(X)$, we mean the arithmetic depth of $X$, i.e. $\depth_{R}(R/I_X)$.
\end{itemize}

\section{Elimination mapping cone construction and Partial elimination ideals}\label{section_2}
In general the mapping cone of the chain map between
two complexes is a kind of natural extension of complexes induced by the
given chain map. Now we construct some graded mapping cone which we call `Elimination mapping cone'. This is naturally related to projections and very useful to understand the syzygies of projections. Another ingredient is the partial elimination ideal theory. Let
us construct the graded mapping cone theorem and consider the partial elimination ideal theory from a viewpoint of inner projections.
\\
\paragraph*{\textbf{Elimination Mapping Cone Construction}}\label{MC}
Let $W=k\langle x_1,\cdots,x_N\rangle \subset V=k\langle x_0,\cdots,x_N\rangle$ be vector spaces over
 $k$ and $S=\Sym(W)=k[x_1,\ldots,x_N]\subset R=\Sym(V)=k[x_0,\ldots,x_N]$ be polynomial rings.
\begin{itemize}
\item  $M$ : a graded $R$-module given a degree 1 shifting map by $\mu$\\ (i.e. $\mu : M_{i}\rightarrow M_{i+1}$)
\item $\Bbb G_{\ast}$(resp. $\Bbb F_{\ast}$) : the graded Koszul complex of $M$, $K_{\ast}^{S}(M)$\\ (resp. $M[-1]$, $K_{\ast}^{S}(M[-1])$) as
follows:
 $$0\rightarrow \wedge^N W\otimes {M}\rightarrow\cdots\rightarrow\wedge^2 W\otimes {M}\rightarrow W\otimes {M}\rightarrow M\rightarrow 0$$
 whose graded components $(\Bbb G_{i})_{i+j}$ are $K_{i}^{S}(M)_{i+j} = \wedge^i W\otimes {M_j}$ \\
 (resp. $(\Bbb F_{i})_{i+j}=\wedge^i W\otimes {M_{j-1}}$).
\item Then $\mu : M_{i}\rightarrow M_{i+1}$ induces the chain map\\ $\bar{\mu}: \Bbb F_{\ast}=K_{\ast}^{S}(M[-1]) \longrightarrow \Bbb G_{\ast}=K_{\ast}^{S}(M)$ of degree 0.
\end{itemize}

Now we construct the mapping cone $(\Bbb{C}_{\bar{\mu}},
\verb"d"_{\bar{\mu}})$ such that
\begin{equation}\label{exact_MC}
0\longrightarrow \Bbb G_{\ast} \longrightarrow
(\Bbb{C}_{\bar{\mu}})_{\ast} \longrightarrow \Bbb
F_{\ast}[-1]\longrightarrow 0\, ,
\end{equation}
where \,$\Bbb{C}_{\bar{\mu}}$ is a direct sum $\Bbb G_{\ast}\bigoplus\Bbb
F_{\ast}[-1]$ and the differential
$\verb"d"_{\bar{\mu}}$ is given by
\[(\verb"d"_{\bar{\mu}})_{\ast}=\left(%
\begin{array}{cc}
  \p_{\Bbb G} & (-1)^{\ast+1}\bar{\mu} \\
  0 & \p_{\Bbb F} \\
\end{array}%
\right),\] where $\p$ is the differential of Koszul complex. From the construction, it can be checked that we have the following
isomorphism (see \cite{AK}):
$$\Tor_i^R(M)_{i+j}\simeq H_i((\Bbb{C}_{\bar{\mu}})_{\ast})_{i+j}\,.$$

Suppose $M$ is a graded $R$-module which is also a graded $S$-module. Consider a multiplication map $\mu = \times x_0$ as a naturally
given degree 1 shifting map on $M$. In this case, the long exact sequence on homology groups induced from (\ref{exact_MC}) is important
and very useful to study the syzygies of projections. Note that in general we can define \textit{property $\textbf{N}_{d,p}$} similarly (i.e. $\beta^R_{i,j}(R/I)=0$ for $any~0\leq i\leq p,~ j\geq d$).

\begin{Thm}{\bf{(Elimination mapping cone sequence)}}\\
Let $S=k[x_1,\ldots,x_N]\subset R=k[x_0,x_1\ldots,x_N]$ be two polynomial rings.
\begin{itemize}
\item[(a)]\label{MC_seq} Let $M$ be a graded $R$-module which is not necessarily finitely generated. Then, we have a natural long exact sequence:
\begin{equation*}
\begin{array}{cccccccc}
\cdots\Tor_{i}^R(M)_{i+j}\lra
\Tor_{i-1}^{S}(M)_{i-1+j}
\sts{\bar{\mu}}\Tor_{i-1}^{S}(M)_{i-1+j+1}\lra\Tor_{i-1}^R(M)_{i-1+j+1}\cdots
\end{array}
\end{equation*}
whose connecting homomorphism $\bar{\mu}$ is induced by the multiplicative map $\times x_0$.
\item[(b)]\label{betti_MC} Assume that $R/I$ satisfies  property $\textbf{N}_{d,p}$ for some $d\ge 2, p\geq 1$. Then a multiplication by $x_0$ induces a sequence of
isomorphisms on $\Tor_{i}^{S}(I)_{i+j}$ for $0\leq i\leq p-2, ~~j\geq d+1$ and a surjection for $j=d$;
$$\cdots\sts{\times x_0}\Tor_{i}^{S}(I)_{i+d}\sst{\times x_0}\Tor_{i}^{S}(I)_{i+d+1}\isost{\times x_0}\Tor_{i}^{S}(I)_{i+d+2}\isost{\times x_0}\cdots.$$
For $i=p-1$, we have a sequence of surjections from $j=d$
$$\cdots\sts{\times x_0}\Tor_{p-1}^{S}(I)_{p-1+d}\sst{\times x_0}\Tor_{p-1}^{S}(I)_{p-1+d+1}\sst{\times x_0}\Tor_{p-1}^{S}(I)_{p-1+d+2}\sst{\times x_0}\cdots$$
\end{itemize}
\end{Thm}

\begin{Remk}
J. Ahn and the second author pointed out that this graded mapping cone construction is closely related to outer projections (see \cite{AK}).
We remark here that this theorem is also true even for an {\it infinitely generated} $S$-module $M$ and relates the torsion module $\Tor^R(M)$
to the torsion module of $M$ as $S$-module. Therefore this gives us useful information about syzygies of {\it inner} projections.
\end{Remk}

\begin{proof} (a) follows from theorem 2.2 in \cite{AK}.
For a proof of (b), consider the mapping cone sequence of Theorem (\ref{MC_seq})
for $M=I$
$$\Tor_{i+1}^R(I)_{i+1+j}\lra
\Tor_{i}^{S}(I)_{i+j}\st{\times x_0}\Tor_{i}^{S}(I)_{i+j+1}\lra
\Tor_{i}^R(I)_{i+j+1}$$ Note that
$\Tor_{i}^R(I,k)_{i+j}=0$ for $0\le i\le p-1$ and $j\geq d+1$ by
assumption that $I$ has $\textbf{N}_{d,p}$ property as a $R$-module. So, we have an
isomorphism
\[\Tor_{i}^{S}(I,k)_{i+j}\isost{\times x_0}\Tor_{i}^{S}(I,k)_{i+j+1}\] for $0\le i\le p-2$, $\forall j\geq d+1$ and a surjection
for $j=d$.

In case $i=p-1$, we know $\Tor_{p-1}^R(I)_{p-1+j}=0$ for $j\geq d+1$ in the mapping cone sequence
$$\Tor_{p}^R(I)_{p+j}\lra\Tor_{p-1}^{S}(I)_{p-1+j}\st{\times x_0}\Tor_{p-1}^{S}(I)_{p-1+j+1}\lra \Tor_{p-1}^R(I)_{p-1+j+1}.$$
Therefore we get the desired surjections for $i=p-1$ case.
\end{proof}

\bigskip

\paragraph*{{\bf Partial Elimination Ideals under a Projection}} 
Mark Green introduced partial elimination ideals in his lecture note \cite{G}. For the degree lexicographic order, if $f\in
I_m$ has leading term $\ini(f)=x_0^{d_0}\cdots x_n^{d_n}$, we set
$d_0(f)=d_0$, the leading power of $x_0$ in $f$. Then we can give
the definition of partial elimination ideals as in the following.

\begin{Def}\label{PEI}
Let $I\subset R$ be a homogeneous ideal and let
\[\widetilde{K}_i(I)=\bigoplus_{m\ge 0}\big\{f\in I_{m}\mid d_0(f)\leq i\big\}.\] If $f\in \widetilde{K}_i(I)$, we may write
uniquely $f=x_0^i\bar{f}+g$ where $d_0(g)<i$. Now we define the ideal $K_i(I)$ in $S$ generated by the image of $\widetilde{K}_i(I)$
under the map $f\mapsto \bar{f}$ and we call $K_{i}(I)$ the \it{$i$-th partial elimination ideal of $I$}.
\end{Def}

\begin{Obs} We can observe some properties of these
ideals in the projection case.
\begin{itemize}
\item[(a)] $0$--th partial elimination ideal $K_{0}(I)$ of $I$ is
\begin{align*}
J:=I\cap S&=\bigoplus_{m\ge 0}\big\{f \in I_m \mid d_0(f)=0\big\}.
\end{align*}
Note that the ideal $J$ gives a scheme structure on the image $X_q$ naturally.
\item [(b)] $\widetilde{K}_{i}(I)$ is a natural filtration of $I$ with respect to $x_0$ which also induces a filtraton on $K_i(I)$'s :
\[J=\widetilde{K}_0(I)\subset \widetilde{K}_1(I)\subset \cdots \subset \widetilde{K}_i(I)\subset \cdots \subset \widetilde{K}_{\infty}(I)=I \]
\[J=K_0(I)\subset K_1(I)\subset \cdots \subset K_i(I)\subset \cdots \subset
S.\]

\item[(c)] $\widetilde{K}_{i}(I)$ is a finitely generated graded $S$-module and there is
a short exact sequence as graded $S$-modules
\begin{equation}\label{ses_PEI}
0\rightarrow \frac{\widetilde{K}_{i-1}(I)}{\widetilde{K}_{0}(I)}
\rightarrow
\frac{\widetilde{K}_{i}(I)}{\widetilde{K}_{0}(I)}\rightarrow
K_{i}(I)(-i)\rightarrow 0.
\end{equation}
\end{itemize}
\end{Obs}

In general we can see at least when the $K_i(I)$'s stabilize and what they look like for inner projections. The following proposition is the anwser. It also tells us a minimal free resoultion for some infinitely generated graded $S$-module which is very useful to understand the defining equations and syzygies of inner projections.

\begin{Prop}\label{linear}
Let $X\subset \P^N$ be a reduced projective scheme with a homogeneous defining ideal $I$. Let $q=(1,0,\ldots,0)\in X$.
\begin{itemize}
\item[(a)] If $I$ satisfies property $\textbf{N}_{d,1}$, $K_i(I)$ stabilizes at least at $i=d-1$ to an ideal defining $TC_{q}X$,
the tangent cone of $X$ at $q$. So if $q$ is smooth, $K_{d-1}(I)$ consists of linear forms which defines $T_q X$.
\item[(b)] In particular, if $I$ is generated by quadrics and $q$ is smooth, then $K_{i}(I)$ stabilizes at $i=1$ step to an ideal
$I_L=(l_1,\ldots,l_e), e=\codim(X,\P^N)$ which defines the tangent space $T_{q}X$, i.e. $\,\, J=K_0(I)\subset
I_L=K_{1}(I)=\cdots=K_{i}(I)=\cdots\subset S$ and $I/J$ has obvious syzygies as an infinitely generated $S$-module such that:
\[
\begin{array}{cccccccccccccccccccc}
&S(-e-1)^{b_e}&& S(-3)^{b_2}&& S(-2)^{b_1}&&&\\
0 \lra & \oplus S(-e-2)^{b_e}&\lra\cdots\lra&\oplus S(-4)^{b_2}&\lra&\oplus S(-3)^{b_1}&\lra&I/J&\lra&0\, ,\\
&\oplus S(-e-3)^{b_e}&&\oplus S(-5)^{b_2}&&\oplus S(-4)^{b_1}&&&&\\
&\cdots&&\cdots&&\cdots&&
\end{array}
\]
where $b_i={e\choose i}$.
\end{itemize}
\end{Prop}

\begin{proof}
(a) Since $I$ is generated in deg $\le d$ and $q=(1,0,\ldots,0)\in X$, we have generators $\{F_i\}$ of $I$ with $d_{0}(F_i)\le d-1$
because there is no generator of the form ${x_0}^{d}+$ other lower terms in ${x_0}$.
From this, every leading term $f$ of a homogeneous polynomial $F$ in $I$ of $\deg k~(k\ge d)$ is written as $x_0^\Box\cdot\bar{f}$
where  $\bar{f}\in K_{c}(I)$ for some $c\le d-1$.
So $K_i(I)$ stabilizes at least at $i=d-1$. Note that all $\bar{f}\in K_{i}(I)\,(i\ge0)$ are also regarded as the defining equations
of tangent cone $TC_{q}X$ of $X$ at $q$ because they come from $f=x_0^{i}\bar{f} + g \in I$, $d_0(g)<i$. Therefore,
$K_i(I)$ stabilizes to the ideal defining $TC_{q}X$. In case of a smooth point $q\in X$, $T_q X=TC_q X$ and $K_{d-1}(I)=(\ell_1,\ldots,\ell_e)$,
$e=\codim(X, \P^N)$. \\
(b) Since $d=2$ and $q$ is a smooth point, $K_{i}(I)$ becomes $I_L=(\ell_1,\ldots,\ell_e)$ for each $i\ge1$. For the sake of the $S$-module syzygy
of $I/J$, first note that $I=\widetilde{K}_{\infty}(I)$. From the exact sequence (\ref{ses_PEI}),
we get $\widetilde{K}_1(I)/J\simeq K_1(I)(-1)$ with the following linear Koszul resolution: letting $b_i={e\choose i}$,
$$0 \lra S(-e-1)^{b_e}\lra\cdots\lra S(-3)^{b_2}\lra S(-2)^{b_1}\lra \widetilde{K}_1(I)/J\lra0 .$$ Next,
$K_{2}(I)(-2)=K_{1}(I)(-2)$ has also linear syzygies:
$$0 \lra S(-e-2)^{b_e}\lra\cdots\lra S(-4)^{b_2}\lra S(-3)^{b_1}\lra K_{2}(I)(-2)\lra 0,$$
and we have the following exact sequence from
(\ref{ses_PEI}) again,
$$0\rightarrow \frac{\widetilde{K}_{1}(I)}{J}\rightarrow \frac{\widetilde{K}_{2}(I)}{J}\rightarrow K_{2}(I)(-2)\rightarrow 0.$$
By the long exact sequence of
$\Tor$, we know that
$$ \begin{array}{cccccccccccccccccccc}
&S(-e-1)^{b_e}&&S(-3)^{b_2}&& S(-2)^{b_1}&&&\\
0 \lra & \oplus S(-e-2)^{b_e}&\lra\cdots\lra&\oplus S(-4)^{b_2}&\lra& \oplus  S(-3)^{b_1}&\lra&\tilde{K}_{2}(I)/J&\lra&0.\\
\end{array} $$
Similarly, we can compute the syzygy of $\widetilde{K}_{i}(I)/J$ for any $i$,
and we get the desired resolution of $I/J=\widetilde{K}_{\infty}(I)/J$ as $S$-module in the end.
\end{proof}

\begin{Remk}\label{commute}
For the next section, we remark some useful facts as follows:
\begin{itemize}
\item[(a)]{\bf(Reduction of syzygies)} From the sequence (\ref{ses_PEI}), we have an isomorphism
\begin{equation}\label{reduction_syz} 
\Tor^S_i(I/J)_{i+j}\simeq\Tor^S_i(\widetilde{K}_{d}(I)/J)_{i+j}~~\textrm{for any}~d\ge j-1 .
\end{equation}
In other word, the syzygies of an infinitely generated $S$-module $I/J$ can be computed from the syzygies of finitely generated $S$-module $\widetilde{K}_i(I)$. Further, if \textit{all $K_i(I)$'s allow only linear syzygies at each step} (i.e. $\Tor^{S}_a(K_i(I))_{a+b}=0$ for $\forall a$ and $\forall b\neq i+1$), then 
\begin{equation*}
\Tor^S_i(I/J)_{i+j}\simeq\Tor^S_i(K_{j-1}(I)(-j+1))_{i+j} ~\textrm{for any}~ i,j~ 
\end{equation*} 
 as Proposition \ref{linear}~(b) shows us that the syzygies of $I/J$ essentially come just from the Koszul syzygies of $\{{x_0}^\alpha\ell_1,\ldots,{x_0}^\alpha\ell_e\}$ of $K_\alpha(I)$.

\item[(b)]{\bf(Commutativity of $x_0$-multiplication)} Consider the $S$-module homomorphism $\phi:I \rightarrow I/J$, the natural quotient map and also consider
 multiplication maps in both $I$ and $I/J$. This multiplication $\times x_0$ is not well-defined in $I/J$, while it is a well-defined
 $S$-module homomorphism in $I$. But if $X$ is quadratic and $q$ is a smooth point, then, by Proposition \ref{linear}~(b) and above (a), we have a commuting diagram in $\Tor$-level:

\xymatrix @R=2pc @C=2pc{
&&\Tor^{S}_{i}(I)_{i+j+1} \ar[r]^{\phi} & \,\Tor^{S}_{i}(I/J)_{i+j+1} \ar[r]^{\psi}_{isom.}&\Tor^S_{i}(L_{j-1}(I))_{i+j+1}\\
&&\Tor^{S}_{i}(I)_{i+j} \ar[r]^{\phi}\ar[u]^{\times x_0} & \Tor^{S}_{i}(I/J)_{i+j} \ar@{-->}[u]^{\exists~\times x_0}\ar[r]^{\psi}_{isom.}&\Tor^S_{i}(L_{j-1}(I))_{i+j}\ar[u]^{\times x_0}~,
} where $L_{j-1}(I):=K_{j-1}(I)(-j+1)$ and each row $\psi\circ\phi$ is induced from the $S$-homomorphism $I \to L_{j-1}(I)$ given by $f = x_0^{j-1}\bar{f}_{j-1} + x_0^{j-2}\bar{f}_{j-2} + \cdots + \bar{f}_{0} \mapsto x_0^{j-1}\bar{f}_{j-1}$ which naturally commutes with the $x_0$-multiplication.
\end{itemize}
\end{Remk}

\begin{Remk}
{\bf{(Outer projection case)}}
\begin{itemize}
\item[(a)] We can also consider outer projection by the similar method. In this case $K_i(I)$ always stabilizes at least at $d-$th step to $(1)=S$
if $I$ satisfies $\textbf{N}_{d,1}$. More interesting fact is that $K_{d-1}(I)$ consists of linear forms with $\textbf{N}_{d,2}$-condition.
Especially, suppose that X satisfy property $\textbf{N}_{2,2}$ and $q=(1,0,\cdots, 0)\notin X$. Then $K_{1}(I)$ is an ideal of linear forms
$I_\Sigma$ defining the singular locus $\Sigma$ of $\pi_{q}$ in ${{X_q}} \subset \P^{N-1}$ (see \cite{AK} for details). By the similar method as in the inner
projection, we see that $I/J$ has simple $S$-module syzygies such that:
\[
\begin{array}{ccccclcccccccccccccc}
0 \lra &S(-t-1)^{b_t}&\lra\cdots\lra& S(-3)^{b_2}&\lra&S(-2)^{b_1+1}&\lra&I/J&\lra&0\, ,\\
&&&&&\oplus S(-3)&&&&\\
&&&&&\oplus S(-4)&&&&\\
&&&&&\cdots&&&&
\end{array}
\]
where $b_i={t\choose i}$, $t=\codim(\Sigma, \P^{N-1})$. So, this resolution can be used to study the outer projection case.
\item[(b)] The stabilized ideal gives an important information for projections. In outer case of $\textbf{N}_{2,p}$ $(p\ge2)$,
it is shown in \cite {AK,P} that the dimension of $\Sigma$ determines the number of quadric equations and the arithmetic
depth of projected varieties according to moving the center of projection. In our inner projection, $K_{1}(I)$ also shows us
the tangential behavior of $X$ at $q$ and $TC_{q}X$ plays an important role in our problem.
\end{itemize}
\end{Remk}

Now there arise some basic and natural questions. \textit{How are the syzygies of $J$ related to the $S$-module syzygies of $I$ and to the $R$-module
syzygies of $I$?} With the assumption for property $\textbf{N}_{2,p}$, we may ask the following question specifically : Is $J$ generated only by quadrics
if so $I$ is? There might be cubic generators like $\ell_iQ_{0,j}-\ell_jQ_{0,i} (=\ell_j\cdot[\ x_0\ell_i-Q_{0,i}]\ - \ell_i\cdot[\ x_0 \ell_j-Q_{0,j}]\ )$
in $J$ (see (\ref{write_I}) in Section \ref{section_1.5}). If not, how about the case of $\textbf{N}_{2,2}$? What can we say about higher linear syzygies of ${{X_q}}$? We will answer these kind of
syzygy and elimination problems and derive stronger results by using the elimination mapping cone sequence and the partial elimination ideal theory in next section.\\


\section{Embedded linear syzygies and Applications}\label{section_3}

Recently, D. Eisenbud et. al. showed that with assumption for some property $\textbf{N}_{2,p}$, the syzygies of $X$ restrict surjectively to the syzygies of linear sections in their paper `Restricting linear syzygies', \cite{EGHP1}.
We consider in this section the behavior of the syzygies under inner projections and we present one of our main theorems on `embedded linear syzygies' which is
the natural projection-analogue of the linear section case.

\begin{Thm}\label{Main result}
Let $X\subset \mathbb P^N$ be a nondegenerate reduced quadratic scheme
whose saturated ideal $I_X$ satisfies property $\textbf{N}_{2,p}$ for some $p\geq 1$ and $q\in X$ be a
smooth point. Consider the inner projection $\pi_{q}:X\dashrightarrow {X_q}\subset \mathbb P^{N-1}$. Then there is an injection between the minimal free resolutions of $I_{{X_q}}$ and $I_X$ up to first $(p-1)$-th step, i.e.
$$\exists\,\,\,
f:\Tor^{S}_{i}(I_{{X_q}},k)_{i+j}\hookrightarrow\Tor^{R}_{i}(I_X,k)_{i+j}\,\,\,
\textrm{for}\,\, 0\le i\le p-2,~~\forall j\in\Z$$
which is induced by the natural inclusion $I_{{X_q}}\hookrightarrow I_X$ and the elimination mapping cone sequence (see Theorem \ref{MC_seq} (a)).
\end{Thm}

\begin{Remk}\label{scheme-theo_Np}
The method used to prove Theorme \ref{Main result} is, in fact, available for any ideal $I$ defining $X$ and $J=I\cap S$ defining $X_q$ scheme-theoretically.
\end{Remk}

\begin{proof} We have a basic short exact
sequence of $S$-modules,
\begin{equation}~\label{idealseq}
0\llra I_{{X_q}} \llra I_X \llra I_X / I_{{X_q}} \llra 0.
\end{equation}

From the long exact sequence of ~(\ref{idealseq}) and the mapping cone sequence of $I_X$ in ~(\ref{MC_seq}), we have a diagram\\[1ex]
\xymatrix @R=3pc @C=2pc{
&0\ar[d]&0\ar[d]&\\
  &\Tor^{S}_{i}(I_{{X_q}} ,k)_{i+j-1} \ar[d] & \Tor^{S}_{i}(I_{{X_q}} ,k)_{i+j}\ar[d]^{g}\ar[dr]^{f\, :=\,h\circ g} &  \\
  &\Tor^{S}_{i}(I_X ,k)_{i+j-1} \ar[d]\ar[r]^{\times x_0} & \Tor^{S}_{i}(I_X ,k)_{i+j} \ar[d]\ar[r]^{h}& **[r]\Tor^{R}_{i}(I_X ,k)_{i+j}\\
  &\Tor^{S}_{i}(I_X /I_{{X_q}} ,k)_{i+j-1} \ar[r]^{\times x_0} & \Tor^{S}_{i}(I_X/ I_{{X_q}} ,k)_{i+j} &
}
\vspace{0.5cm}
For any $0\le i\le p-2$, we proceed with $j$ case by case.\\

\verb"Case 1)" $j\le 1$: Since $\Tor^{S}_{i}(I_{{X_q}} ,k)_{i+j}=0$, so it is obviously injected to $\Tor^{R}_{i}(I_X ,k)_{i+j}$ by $f$.\\

\verb"Case 2)" $j=2$ (i.e. linear syzygy cases for each $i$): From (\ref{idealseq}), we have
$$\Tor_{i+1}^{S}(I_X /I_{{X_q}},k)_{i+2}\llra\Tor_{i}^{S}(I_{{X_q}},k)_{i+2}\st{g}\Tor_{i}^{S}(I_X,k)_{i+2}$$
Since $q$ is a smooth point, with $N_{2,1}$ condition we know the syzygy structures of $I_X /I_{{X_q}}$ as $S$-module by Proposition \ref{linear}~(b).
It shows $\Tor^{S}_{i+1}(I_X /I_{{X_q}} ,k)_{i+2}=0$, implying that $g$ is injective. Since $X$ is nondegenerate, so $\Tor^{S}_{i}(I_X,k)_{i+1}=0$ and $h$ is also
injective at the horizontal mapping cone sequence of above diagram. Hence $f$ is injective in this case, too.\\

\verb"Case 3)" $j\ge 3$:  First note that $\Tor^{R}_{i}(I_X,k)_{i+j}=0$ for $0\le i \le p-1, j\ge 3$ by the assumption of property $\textbf{N}_{2,p}$.
We will show that $g$ is injective and $\Tor^{S}_{i}(I_X,k)_{i+j}$ is isomorphic to $\Tor^{S}_{i}(I_X /I_{{X_q}},k)_{i+j}$ for $0\le i \le p-2$. Then we can conclude that $\Tor^{S}_{i}(I_{{X_q}},k)_{i+j}=0$, so $f$ is injective for $0\le i \le p-2$.

Consider the commutative diagram (by Remark \ref{commute}) in the third quadrant part of above diagram. Repeating this diagram by multiplying $x_0$ sufficiently, we have the following diagram:\\[1ex]
\xymatrix @R=2.5pc @C=2.5pc{
\Tor_{i+1}^{S}(I_X,k)_{i+1+j-1}\ar@{->>}[r]^-{\therefore}_-{surj.} \ar@{->>}[d]^-{surj.}&\Tor_{i+1}^{S}(I_X/I_{{X_q}},k)_{i+1+j-1}\ar[r] \ar[d]^-{isom.}&\Tor_{i}^{S}(I_{{X_q}},k)_{i+j}\ar[r]^-{g}& \\
\Tor_{i+1}^{S}(I_X,k)_{i+1+N}\ar[r]^-{isom.} &\Tor_{i+1}^{S}(I_X/I_{{X_q}},k)_{i+1+N} &&
}
\vspace{0.5cm}

The left vertical map is surjective from Theorem~\ref{betti_MC}~(b), and the right one is an isomorphism by the syzygy structures of $I_X/I_{{X_q}}$ in Proposition \ref{linear}~(b). Since $I_{{X_q}}$ is a finite $S$-module, $\Tor^{S}_{i}(I_{{X_q}}, k)_{i+N}=0$ for sufficiently large $N$, so we get the below(second row) isomorphism. Therefore the map $$\Tor_{i+1}^{S}(I_X,k)_{i+1+j-1}\lra\Tor_{i+1}^{S}(I_X/I_{{X_q}},k)_{i+1+j-1}$$ is surjective, and $g$ is injective.

Similarly, we can have the desired isomorphism between $\Tor^{S}_{i}(I_X,k)_{i+j}$ and $\Tor^{S}_{i}(I_X /I_{{X_q}},k)_{i+j}$ as follows:\\[1ex]

\xymatrix @R=2.5pc @C=2.5pc{
&\Tor_{i}^{S}(I_{{X_q}},k)_{i+j}\ar[r]^-{g}&\Tor_{i}^{S}(I_X,k)_{i+j}\ar[r]^-{\alpha}_-{isom.} \ar[d]^-{isom.}&\Tor_{i}^{S}(I_X/I_{{X_q}},k)_{i+j} \ar[d]^-{isom.} \\
&&\Tor_{i}^{S}(I_X,k)_{i+N}\ar[r]^-{isom.} &\Tor_{i}^{S}(I_X/I_{{X_q}},k)_{i+N}
}
\vspace{0.5cm}

In this case, the mapping cone construction gives the left vertical isomorphism by Theorem \ref{betti_MC}~(b). So the above map $\alpha$ is an isomorphism as we wish, and $\Tor^{S}_{i}(I_{{X_q}},k)_{i+j}=0$ for $0\le i \le p-2, j\ge 3$.
\end{proof}

This main Theorem \ref{Main result} tell us that all the $S$-module syzygies of ${{X_q}}$ are exactly the very ones which are \textit{already} embedded in the linear syzygies of $X$ as $R$-module. This doesn't hold for outer projection and inner projection of varieties with $\textbf{N}_{d,p}\,(d\ge3)$.

\begin{Ex}
Let $C$ be a rational normal curve in $\P^3$ and $I_C$ be the homogeneous ideal $(x_0x_2-x_1^2, x_0x_1-x_1x_3-x_2^2,x_0^2-x_0x_3-x_1x_2)$ under suitable coordinate change. We know $C$ is 2-regular and consider an outer projection of $C$ from $q=(1,0,0,0)$. Then $I_{C_q}=(x_1^3-x_1x_2x_3-x_2^3)$ has a cubic generator (i.e. $\textbf{N}_{3,1}$). Since $x_1^3-x_1x_2x_3-x_2^3=(-x_1)\cdot [\ x_0x_2-x_1^2 ]\ +x_2\cdot [\ x_0x_1-x_1x_3-x_2^2 ]\ $, this is zero in $\Tor^R_0(I_C)_3$ and $\Tor^S_0(I_{C_q})_3\lra\Tor^R_0(I_C)_3$ is not injective. In general, if we take the center $q\in L\simeq \P^2$ which is a multisecant (\textit{e.g. at least $4$-secant}) $2$-plane, then for outer and inner projection cases there is a multisecant line to $X_q$. So, the defining equations of $X_q$ may have larger degrees.
\end{Ex}

As an immediate consequence, we have the following corollary.
\begin{Coro}{\bf(Property $\textbf{N}_{2,p-1}$ of inner projections)}\label{Main result N_{2,p}}
Let $X\subset \mathbb P^N$ be a nondegenerate reduced quadratic scheme satisfying property $\textbf{N}_{2,p}$ for some $p\geq 2$ and $q\in X$ be a
smooth point. Then, the inner projection $X_q$ is also quadratic and satisfies property $\textbf{N}_{2,p-1}$.
\end{Coro}
\begin{proof}
\verb"Case 3)" $j\ge 3$ in the proof of Theorem ~\ref{Main result} is a proof of this corollary.
\end{proof}

\begin{Remk}\label{remk_for_N2,p} There are some remarks on Corollary \ref{Main result N_{2,p}}.
\begin{itemize}
\item[(a)] This corollary can be easily extended to the case of a \textit{general} linear subspace $\Lambda\simeq\P^t$ such that $\dim X\cap\Lambda$ is zero. Precisely, \textit{if $\Lambda$ does not meet $Sing(X)$ and $\langle \Lambda\cap X \rangle =\Lambda$}, then the $t+1$ points of $\Lambda\cap X$ are in linearly general position so that $X_{\Lambda}$ satisfies $\textbf{N}_{2,p-t-1}$ by successive inner projections. To complete this question in Problem (a) for \textit{any} linear subspace $\Lambda$, it remains to consider that how the projections from a singular center or a linear subvariety contained in $X$ behave (see Question \ref{Q_sing}, \ref{Q_subvar}).
\item[(b)] For a smooth irreducible variety $X\subset \Bbb P(H^0(\mathcal L))$ with the condition $\textbf{N}_p$ $(p\ge 1)$ embedded by the complete linear system of a
very ample line bundle $\mathcal L$ on $X$, Y. Choi, P. Kang and S. Kwak showed that the inner projection $X_q$ is smooth and satisfies property $\textbf{N}_{p-1}$
for any $q\in X\setminus{\Trisec}(X)$, i.e. property $\textbf{N}_{p-1}$ holds for $({\Bl}_q(X),\sigma^*\mathcal L-E)$ by using vector bundle techniques and Koszul cohomology methods due to Green-Lazarsfeld (see \cite{CKK}).
Our Corollary \ref{Main result N_{2,p}} extends this result to the category of reduced projective schemes satisfying property $\textbf{N}_{2,p}$ with any smooth point $q\in X$. Note that this uniform behavior looks unusual in a sense that linear syzygies of outer projections heavily depend on moving the center of projection in an ambient space $\P^N$ (see \cite {CKP,KP,P}).
\end{itemize}
\end{Remk}

In order to understand the Betti table of inner projections, we need to consider defining equations of inner projections, depth, and the Castelnuovo-Mumford regularity.

\begin{Prop}{\bf(Quadratic equations of inner projections)}\label{quad_dim}\\
Let $X \subset \P^N$ be a nondegenerate reduced scheme with a defining ideal $I$ and $any$ (possibly $singular$) point $q \in X$. For the inner projection ${X_q}\subset \P^{N-1}$, we have
\begin{itemize}
\item[(a)] $\beta_{0,2}(J) = \beta_{0,2}(I) - \beta_{0,1}(K_1(I))$, where $J$ is the $x_0$-elimination ideal of $I$ as usual. 
\item[(b)] Furthermore, if $I$ is quadratic (so $X$ is quadratic), then we have $\beta_{0,2}(J)=  \beta_{0,2}(I)- N + \dim T_q X$. In particular, in case of $I=I_X$ it coincides with $$h^0(\P^{N-1}, \mathcal I_{{{X_q}}}(2))= h^0(\P^{N}, \mathcal I_{X}(2))- N + \dim\, T_{q}X~.$$
\end{itemize}
\end{Prop}

\begin{proof}
As in the proof of Theorem \ref{Main result}, there is a long exact sequence such that
$$\to\Tor^S_1(I/J,k)_2\to\Tor^S_0(J,k)_2\to\Tor^S_0(I,k)_2\to\Tor^S_0(I/J,k)_2\to 0~.$$
From the reduction of syzygies (\ref{reduction_syz}) in Remark \ref{commute}, we have
\begin{equation*}
\begin{cases}
\Tor^S_1(I/J,k)_2\simeq\Tor_{1}^{S}(K_1(I)(-1),k)_2=\Tor^S_1(K_1(I),k)_1=0\\
\Tor^S_0(I/J,k)_2\simeq\Tor_{0}^{S}(K_1(I)(-1),k)_2=\Tor^S_0(K_{1}(I),k)_1~,
\end{cases}
\end{equation*}
which implies the desired formula in (a) directly.

If $I$ is a quadratic ideal, then by Proposition \ref{linear} the $K_i(I)$ stabilizes at $i=1$ and $K_1(I)$ defines the tangent cone $TC_q X$. We may write $$K_1(I)=(\ell_1,\cdots,\ell_t,\text{\,higher\,\,degree\,\,polynomials\,})$$ , where $I=(x_0 \ell_1-Q_{0,1},\cdots,x_0 \ell_t-Q_{0,t},Q_1,\cdots,Q_s)$ just as our convention. Among the elements of $K_1(I)$ the generators of $K_1(I)_1$, $\{\ell_1,\cdots,\ell_t\}$ define the tangent space $T_q X$ so that we have $\beta_{0,1}(K_1(I))=\codim(T_q X,\P^N)$.
\end{proof}

\begin{Remk}
In outer projection case, there is a formula $h^0(\mathcal I_{{{X_q}}}(2))=h^0(\mathcal I_X(2))-(N-\dim\Sigma_q (X))$ if $X$ satisfies property $\textbf{N}_{2,2}$ (see proposition 4.11 in \cite{AK}, theorem 3.3 in \cite{P}). This also shows that there is a tendency of having more quadrics for projected varieties as $q$ is getting closer to $X$. Note that the negative value of $h^0(\mathcal I_{X_q}(2))$ implies that there is no quadric vanishing on $X_q$.
By this fact, we can expect that the inner projection case has more linear syzygies as Corollary \ref{Main result N_{2,p}} shows.
\end{Remk}

Next question is that how many quadrics defining $X$ are required to satisfy property $\textbf{N}_{2,p}$ and we give the sharp lower bound in the following.
\begin{Coro}{\bf(Neccesary lower bound for property $\textbf{N}_{2,p}$)}\label{LB}
Let $X$ be a nondegenerate reduced quadratic scheme in $\P^{r+e}$ of codimension $e$ and $I$ be the quadratic ideal of $X$.
Suppose that $I$ satisfies property $\textbf{N}_{2,p}$ and $\beta_{0,2}(I)$ is the number of generators of $I$. Then $\beta_{0,2}(I)$ is not less than $\textrm{LB}_p$
as follows:
$$\textrm{LB}_{p}=e\cdot p -\frac{p(p-1)}{2}\le \beta_{0,2}(I) \le \beta_{0,2}(I_X)\,\,(=h^0(\mathcal I_{X}(2))\,)$$
\end{Coro}

\begin{proof} Let's take a smooth point $q_0$ in $X$ and project $X$ from $q_0$. Let $X^{(1)}$ be the image (the Zariski closure)  and $I^{(1)}$ be the elimination ideal of $I$. Then, from Proposition \ref{quad_dim} we get
$$\beta_{0,2}(I^{(1)})=\beta_{0,2}(I)-(r+e)+r.$$
We also know that $I^{(1)}$ defines $X^{(1)}$ scheme-theoretically and satisfies property $\textbf{N}_{2,p-1}$. Take another smooth point $q_1$ in $X^{(1)}$ and project it from $q_1$. Then, with the same notation, we have
$$\beta_{0,2}(I^{(2)})=\beta_{0,2}(I^{(1)})-(r+e-1)+r.$$ Taking successive inner projections, we get
$$\beta_{0,2}(I^{(p-1)})=\beta_{0,2}(I^{(p-2)})-(r+e-p+2)+r.$$
Summing up both sides of above equations, it gives
$$\beta_{0,2}(I^{(p-1)})=\beta_{0,2}(I)-\frac{(p-1)(2r+2e-p+2)}{2}+r(p-1) \cdots(\ast)~.$$
And we know that $X^{(p-1)}$ is still cut by quadrics (i.e. $\textbf{N}_{2,1}$). So $\beta_{0,2}(I^{(p-1)})$ is not less than $\codim\,X^{(p-1)}=(r+e)-p+1-r=e-p+1$. If we plug-in this inequality to $(\ast)$, we've got the desired bound $\textrm{LB}_{p}$.
\end{proof}

\begin{Remk}\label{LB_remk}
This bound is sharp for $p=1$ by complete intersections, $p=e-1$ by \textit{del Pezzo} varieties (see Theorem \ref{rigidity} (b)), and $p=e$ by minimal degree varieties. Note also that the upper bound for $\beta_{0,2}(I_X)$ for a nondegenerate integral subscheme $X\subset \P^{r+e}$ of codimension $e$ is $\frac{e(e+1)}{2}$ and this maximum number can be attained if and only if the variety $X$ is of minimal degree from Corollaries $5.4, 5.8$ in \cite{Z}.
\end{Remk}

\begin{Remk}{\bf(Degree bound by property $\textbf{N}_{2,p}$)}\label{Al-Si} Recently, A. Alzati and J.C. Sierra get a bound of quadrics for $\textbf{N}_{2,2}$ as paying attention to the structures of the rational map associated to the linear system of quadrics defining $X$, which coincides with our bound $\textrm{LB}_2$ (see \cite{AS}). They also derive a degree bound in terms of codimension $e$, ${d \choose 2}\le {2e-1\choose e-1}$ whose asymptotic behavior is $2^e/\sqrt[4]{\pi e}$ and describe the equality condition: \textit{this holds if and only if the equality of} $\textrm{LB}_2$ \textit{holds}. From this theorem, in case of $p=2$ we also have a rigid condition on degree of the boundary $X$ as if we get some rigidity when $p=1,e-1$, and $e$ (see Remark \ref{LB_remk} and Theorem \ref{rigidity}). Using our inner projection method (\textit{e.g.} Corollary \ref{LB}), we could improve this degree bound a little as follows:
$${d+2-p \choose 2}\le {2e+3-2p\choose e+1-p} \;\;,\;\; d\thicksim 2^{e+2-p}/\sqrt[4]{\pi e} \;\;\;(\textrm{as}\;\; e\;\;\textrm{getting sufficiently large})$$
under the assumption of property $\textbf{N}_{2,p}~(p\ge 2)$ of $X$.
\end{Remk}

\begin{Ex}
It would be interesting to know that if $e\le \beta_{0,2}(I_X)< 2e-1$ then $X$ has always at least a syzygy of defining equations which is not linear because property $\textbf{N}_{2,2}$
does not hold for $X$. For example, let $C$ be the general embedding of degree 19 in $\P^7$ of genus $12$. Then $C$ is a smooth arithmetically Cohen-Macaulay curve  which is cut out scheme-theoretically by  9 quadrics, but the homogeneous ideal $I_C$ is generated by $9$ quadrics and $2$ cubics (see \cite {Ka} for details). These quadratic generators should have at least a syzygy of higher degree as well as linear syzygies.
\end{Ex}

\begin{Ex}{\bf(Veronese embedding $v_d(\P^n)$)}
It is shown that $v_d(\P^n)$ fails property $\textbf{N}_{2,3d-2}$ for $n\ge 2$ and $d\ge 3$ and conjectured that $v_d(\P^n)$ satisfies property $\textbf{N}_{2,3d-3}$ for $n\ge 2$ and $d\ge 3$ (see \cite{OP, EGHP1}). We can also verify the failure of property $\textbf{N}_{2,p}$ of Veronese embedding $X=v_d(\P^n)$ for some cases by using this low bound $\textrm{LB}_{p}$. For example, when $n=2, d=3, p=3d-2=7$, we get $\beta_{0,2}(I_X)=27$ and $\textrm{LB}_7=7\cdot7-{7\choose2}=28$. Therefore, $v_3(\P^2)$ fails to satisfy $\textbf{N}_{2,7}$. Similarly, $v_2(\P^3)$ fails property $\textbf{N}_{2,6}$.
However, it does not give the reason why $v_3(\P^3)$ does fail to be $\textbf{N}_{2,7}$ for the case $n=d=3, p=3d-2=7, e=16$, because $\beta_{0,2}(I_X)=126>91=\textrm{LB}_{7}$. For such $p$ in the middle area of $1\le p\le e$, $\textrm{LB}_p$ seems not to give quite sufficient information for property $\textbf{N}_{2,p}$, while it may be more effective to decide $\textbf{N}_{2,p}$ of a given variety for rather large $p$ among $1\le p\le e$.
\end{Ex}

\section{Arithmetic depth and Syzygetic rigidity}\label{section_3.5}

Now, we proceed to investigate the $\depth$ of inner projections to understand the shape of the Betti table and Castelnuovo-Mumford regularity.
In this section we always consider the saturated ideal $I_X$ among ideals defining $X$, because $\depth_R(R/I)=0$ for any defining ideal $I$ which is not saturated. The following result looks very surprising when we compare this with the outer projection case (see Remark \ref{depth_remk}).

\begin{Thm}\label{depth}{\bf(The depth of inner projections)}
Let $X\subset \mathbb P^N$ be a nondegenerate reduced subscheme and $I_X$ be generated by quadrics. Consider the inner projection $\pi_{q}:X\dashrightarrow {X_q}\subset \P^{N-1}$ from a smooth point $q\in X$. Then,
\begin{itemize}
\item[(a)] the projective dimension of $S/I_{{X_q}}$, $\pd_{S}(S/I_{{X_q}})=\pd_{R}(R/I_X)-1$;
\item[(b)] $\depth_{R}(X)=\depth_{S}({{X_q}})$. In particular, $X$ is arithmetically Cohen-Macaulay if and only if so is ${{X_q}}$.
\end{itemize}
\end{Thm}

\begin{proof} (a)
We know $\pd_{R}(R/I_X)\ge e=\codim X$. Let $l$ be $\pd_{R}(R/I_X)$ (so, $l\ge e$), and $j_0=$max$\{j|\,\Tor^{R}_{l}(R/I_X)_{l+j}\neq0\}$.\\[1ex]
\verb"Case 1)"  Non-Cohen Macaulay case (i.e. $l=e+\alpha,~~~ \alpha\ge1$) :

First of all, we have the following diagram from the exact sequence (\ref{idealseq}):
$$\begin{array}{lccccccccc}
i\,\,\,of\,\,\Tor^{S}_{i}(S/I_{{X_q}})&&0\lra&I_{{X_q}}&\lra&I_X& \lra&I_X/I_{{X_q}}\lra0\\
&&&\uparrow&&\uparrow&&\uparrow&\\
1&&&\Box&&\Box&&S(-2)^{e}\oplus S(-3)^e\oplus\cdots&\\
2&&&\vdots&&\vdots&&\vdots&\\
&&&\uparrow&&\uparrow&&\uparrow&\\
e&&&\Box&&\Box&&S(-e-1)\oplus S(-e-2)\oplus\cdots&\\
&&&\uparrow&&\uparrow&&\uparrow&\\
e+1&&&\Box&\cong&\Box&&0&\\
\vdots&&&\vdots&\vdots&\vdots&&\vdots&\\
l=e+\alpha&&&\blacksquare&\cong&\blacksquare&&0&\\
\blacksquare:vanished&&&&&&&&\\
\end{array}$$

From this diagram, we get $\Tor^{S}_l(R/I_X)\cong\Tor^{S}_l(S/I_{{X_q}})$ as finite $k$-vector spaces. Since $\Tor^{R}_{l+1}(R/I_X)_{l+1+j}=0$
for all $j$ ($\because \pd_{R}(R/I_X)=l$) and $\Tor^{R}_{l}(R/I_X)_{l+j}=0$ for $j> j_0$, we can observe using the mapping cone sequence (\ref{MC_seq}) that
$$\cdots\ist{\times x_0}\Tor^{S}_l(R/I_X)_{l+j}\ist{\times x_0}\cdots\ist{\times x_0}\Tor^{S}_l(R/I_X)_{l+j_0}\isost{\times x_0}\cdots$$
So we have $\Tor^{S}_l(R/I_X)=0$, because it is finite dimensional. This means $\Tor^{S}_{l}(S/I_{{X_q}})\cong\Tor^{S}_{l}(R/I_X)=0$.\\

Next, we claim that $\Tor^{S}_{l-1}(S/I_{{X_q}})\neq0$, which implies $\pd_{S}(S/I_{{X_q}})=l-1$. If $\alpha\ge 2$, then we have $\Tor^S_{l-1}(S/I_{X_q})\simeq\Tor^S_{l-1}(R/I_X)$. Since $\Tor^{S}_{l}(R/I_X)=0$,  we have a nontrivial kernel of the $\times x_0$ map in $I_X$
from the mapping cone sequence (\ref{MC_seq})
$$\begin{array}{ccccccccc}
0\lra&\Tor^{R}_l(R/I_X)_{l+j_0}\hookrightarrow&\Tor^{S}_{l-1}(R/I_X)_{l-1+j_0}\sts{\times x_0}&\Tor^{S}_{l-1}(R/I_X)_{l-1+j_0+1}&\cdots(\ast)\\
&\nparallel&\parallel\rlap{$\wr$}&\parallel\rlap{$\wr$}&&\\
&0&\Tor^S_{l-2}(I_X)_{l-1+j_0}&\Tor^S_{l-2}(I_X)_{l-1+j_0+1}&&\\
\end{array}$$

This implies $\Tor^S_{l-1}(R/I_X)\simeq\Tor^S_{l-1}(S/I_{X_q})\neq 0$ as wished. So, let us focus on the case $\alpha=1$ and so, $l=e+1$.
Consider the following sequence and commutative diagram:
$$\Tor^{S}_{l-1}(I_X/I_{{X_q}})=0\rightarrow\Tor^{S}_{l-2}(I_{X_q})\rightarrow \Tor^{S}_{l-2}(I_X)\rightarrow\Tor^{S}_{l-2}(I_X/I_{{X_q}})\rightarrow\cdots$$
$$
\begin{array}{cccccccccccccccccccc}
\Tor_{l-2}^{S}(I_X)_{e+j_0}\simeq & S(-e-j_0)^{\Box}\otimes k & \sts{f_{e+ j_0}} & S(-e-j_0)\otimes k&\simeq\Tor^{S}_{l-2}(I_X/I_{{X_q}})_{e+j_0}\\
& \llap{$h$}\Big\downarrow\rlap{{not injective}} &&
\llap{$g$}\Big\downarrow{\rlap{$\wr$}} \\
& S(-c)\otimes k & \stackrel{\sim}{\lra} & S(-c)\otimes k\\
\end{array}
$$
where $h,g$ are induced by the multiplication of $x_{0}^{n}$ and $g$ is an isomorphism.
To check $\Tor^{S}_{l-1}(S/I_{{X_q}})\cong\Tor^{S}_{l-2}(I_{{X_q}})\neq0$,
it is enough to show that $f:\Tor^{S}_{l-2}(I_X)\lra\Tor^{S}_{l-2}(I_X/I_{{X_q}})$ is not injective because $\Tor^{S}_{l-1}(I_X/I_{{X_q}})=0$.
Now let me explain why $f$ be not injective. We get the below isomorphism map for $c\gg0$ because $\Tor^{S}_i(I_{{X_q}})$
are finite-dimensional graded vector spaces and also $h$ is not injective by $(\ast)$. From Remark (\ref{commute}), this diagram
commutes and $f_{e+ j_0}$ has a nontrivial kernel. Hence $f:\Tor^{S}_{l-2}(I_X)\lra\Tor^{S}_{l-2}(I_X/I_{{X_q}})$ is not injective and  $\Tor^S_{l-1}(S/I_{X_q})_{e+j_0}=\Tor^{S}_{l-2}(I_{X_q})_{e+j_0}\neq0$.\\[1ex]

\verb"Case 2)"  Cohen-Macaulay case (i.e. $l=e,~~ \alpha=0$) :
In this case, we have the long exact sequence on $\Tor$ as follows:
$$
\begin{array}{cccccccccccccccccccc}
&&&&&&S(-e-1)\oplus\\
0=\Tor^{S}_{e}(I_X/I_{{X_q}})&\lra&\Tor^{S}_{e-1}(I_{{X_q}})&\lra&\Tor^{S}_{e-1}(I_X)&\st{f}&S(-e-2)\oplus\\
&&&&&&\cdots\\
\end{array}
$$
Since $\pd_{R}(R/I_X)=e$, $\Tor^{R}_{e+1}(R/I_X)=0$ and we have an injection
$$\Tor^{S}_{e-1}(I_X)_{e+j}=\Tor^{S}_{e}(R/I_X)_{e+j}\st{\times x_0^n}\Tor^{S}_{e}(R/I_X)_{e+j+n}=\Tor^{S}_{e-1}(I_X)_{e+j+n}$$
for any $j,n$ from the mapping cone sequence (\ref{MC_seq}). By almost same argument using the commuting diagram as \verb"Case 1)", $\alpha=1$,
we can conclude that $f:\Tor^{S}_{e-1}(I_X)\lra\Tor^{S}_{e-1}(I_X/I_{{X_q}})$ is injective and $\Tor^{S}_{e}(S/I_{{X_q}})=0$.
So, this means $\pd_{S}(S/I_{{X_q}})\le e-1$.
But we know $\pd_S(S/I_{X_q}) \ge \codim({{X_q}})=e-1$, therefore $\pd_{S}(S/I_{{X_q}})=e-1$.
On the other hand, (a) implies that $\depth_{R}(X)=\depth_{S}({{X_q}})$ by Auslander-Buchsbaum formula.
\end{proof}

\begin{Remk}\label{depth_remk}
Let $X\subset \P^n$ be a reduced scheme satisfying property $\textbf{N}_{2,p} ~(p\ge 2)$.
Let $\Sigma_q(X)=\{ x\in X \,|\,{\pi_q}^{-1}(\pi_q(x)) \text{ has length }\geq 2\}$ is the \textit{secant locus} of the outer projection.
We would like to remark that $\depth(X_q)=\min\{\depth(X), \dim\Sigma_q(X)+2\}$ for a smooth scheme $X$ (see \cite {AK,P}).
On the other hand, it would be interesting to ask the following question: Is there an example such that $\depth(X_q)\neq \depth(X)$
for inner projections?
\end{Remk}

As an interesting application of above results, we can also prove the following rigidity theorem for the \textit{extremal} ($i.e.~p=e$) and \textit{next to extremal} ($i.e.~p=e-1$) cases of property $\textbf{N}_{2,p}$ of the varieties by using Corollary \ref{Main result N_{2,p}}, Proposition \ref{quad_dim} and Theorem ~\ref{depth}.

\begin{Thm}{\bf(Syzygetic rigidity on property $\textbf{N}_{2,p}$})\label{rigidity}
Let $X$ be a nondegenerate $r$-dimensional variety (i.e. irreducible, reduced) in $\P^{r+e}$,  $e=\codim(X,\P^{r+e})$.
\begin{itemize}
\item[(a)](extremal case) $X$ satisfies property $\textbf{N}_{2,e}$ if and only if $X$ is a minimal degree variety, i.e. a whole linear space $\P^{r+e}$, a quadric hypersurface, a cone of the Veronese surface in $\P^5$, or rational normal scrolls;
\item[(b)](next to extremal case) $X$ fails property $\textbf{N}_{2,e}$ but satisfies $\textbf{N}_{2,e-1}$ if and only if $X$ is a del Pezzo variety, i.e. $X$ is arithmetically Cohen-Macaulay (\textit{abbr.} ACM) and is of next to minimal degree.
\end{itemize}
\end{Thm}
\begin{proof}
Let $\delta(X):=\deg(X)-\codim(X)$ for any subvariety $X\subset \P^{r+e}$. Note that $\delta(X_q)=\delta(X)$ under an inner projection from a smooth point $q\in X$. Take successive inner projections from smooth points. Call the images (Zariski closure) $X=X^{(0)}, X^{(1)},\ldots,X^{(e-2)}$ and $I^{(i)}$ for the elimination ideal of $I^{(i-1)}$ cutting out $X^{(i)}$ scheme-theoretically. By Corollary \ref{Main result N_{2,p}} we know that this $X^{(e-2)}$ has $\codim~2$ and have property $\textbf{N}_{2,2}$ for (a) ~(and $\textbf{N}_{2,1}$ for (b), respectively). Because $X^{(e-1)}$ is a hypersurface, by Proposition \ref{quad_dim} the possible $\beta_{0,2}(I^{(e-2)})=2$ or $3$.
For the case (a), take an inner projection once more and then $X^{(e-1)}$ still satisfies property $\textbf{N}_{2,1}$, i.e an (irreducible) quadric hypersurface. So, $$\beta_{0,2}(I^{(e-1)})=1,\,  \beta_{0,2}(I^{(e-2)})=1+2=3 ~\textrm{and}~ \delta(X)=\delta(X^{(e-1)})=1 .$$ That is, $X$ is of minimal degree.
In the case of (b), $X^{(e-2)}$ is a complete intersection of two quadrics in $\P^{r+2}$ and $X^{(e-1)}$ should be a cubic hypersurface.

In particular, the projective dimension of $X^{(e-2)}$ is equal to $2=\pd_{R}(R/I_X)-(e-2)$ by Theorem ~\ref{depth}.
Therefore,
$$\beta_{0,2}(I^{(e-1)})=0, \, \beta_{0,2}(I^{(e-2)})=2,\, \delta(X)=\delta(X^{(e-1)})=2 ~\textrm{and}~ \pd_{R}(R/I_X) = e,$$
which means $X$ is arithmetically Cohen-Macaulay and of next to minimal degree with $H^0(\mathcal I_X(2))=\frac{(e+2)(e-1)}{2}$.
By the well-known classification of varieties of next to minimal degree, $X$ is a del Pezzo variety. On the other hand, the curve section $C$ of a del Pezzo variety
$X$ is either an elliptic normal curve or a projection of a rational normal curve from a point in $\Sec(C)\setminus C$.  Since $X$ and $C$ have the same Betti table, $X$ satisfies property $\textbf{N}_{2,e-1}$ but fail property $\textbf{N}_{2,e}$.
\end{proof}

\begin{Remk}\label{rem_rig} There are some remarks on Theorem \ref{rigidity}.
\begin{itemize}
\item[(a)] For the smooth projectively normal variety $X$, M. Green's $\mathcal{K}_{p,1}$-theorem in \cite{G2} gives a necessary condition on Theorem \ref{rigidity} (b) (i.e. $X$ is either a variety of next to minimal degree \textit{or} a divisor on a minimal degree). Using our Corollary \ref{Main result N_{2,p}} and \textit{Depth} theorem \ref{depth}, we could obtain the results on both $\deg(X)=\codim X +2$ and ACM property and show the rigidity on next to extremal case for \textit{any} (not necessarily projectively normal) variety $X$.
\item[(b)] Classically, normal del Pezzo varieties were classified by T. Fujita in \cite {F}. And every non-normal del Pezzo variety $X$ (see \cite{BS,F}) comes from outer projection of a minimal degree variety $\widetilde{X}$ from a point $q$ in $\Sec(\widetilde{X})\setminus \widetilde{X}$ satisfying $\dim\Sigma_q(\widetilde{X})=\dim\widetilde{X}-1$ (see remark \ref{depth_remk}). Note that the dimension of the secant locus $\Sigma_q(\widetilde{X})$ varies as $q$ moves in $\Sec(\widetilde{X})\setminus\widetilde{X}$. Thus one can try to classify the non-normal del Pezzo varieties by the type of the secant loci $\Sigma_q(\widetilde{X})$. This is recently classified in \cite{BP} such that there are only $8$ types of non-normal del Pezzo varieties which are not cones. For example, we find projections of a smooth cubic surface scroll $S(1,2)$ in $\P^4$ from any $q\in\P^4\setminus S(1,2)$ or projections of a smooth 3-fold scroll $S(1,1,c)$ in $\P^{c+4}$ with $c>1$ from any $q\in\langle S(1,1)\rangle\setminus S(1,1,c)$, etc.
\end{itemize}

\end{Remk}

Furthermore, let's consider the followig category. (We say that an algebraic set $X=\cup X_i$ is \textit{connected in codimesion 1} if all the component $X_i$'s can be arranged in such a way that every $X_i\cap X_{i+1}$ is of codimension 1 in $X$.)
$$\{\text{equidimensional, connected in codimension 1, reduced subschemes in}~ \P^{r+e}\}\cdots(\ast)$$

In the category $(\ast)$, we have natural notions of $\dim X$ and $\deg(X)$ which is the sum of degrees of $X_i$'s. And as in the category of \textit{varieties}, we also have the `basic' inequality of degree, i.e.$~\deg(X)\ge \codim X + 1$, so it is worthwhile to think of `minimal degree' or `next to minimal degree' in this category.

Using same methods, the Theorem \ref{rigidity}  can be easily extended for this category. We call $X$ (or the sequence) \textit{linearly joined} whenever all the components can be ordered $X_1,X_2,\ldots,X_k$ so that for each $i
$, $(X_1\cup\cdots\cup X_i)\cap X_{i+1}=\textrm{span}(X_1\cup\cdots\cup X_i)\cap\textrm{span}(X_{i+1})$. Then, we have a corollary as follows:

\begin{Coro}\label{rigidity_2}
Let $X$ be a nondegenerate subscheme in the category $(\ast)$ with $e=\codim(X,\P^{r+e})$.
\begin{itemize}
\item[(a)](extremal case) $X$ satisfies property $\textbf{N}_{2,e}$ if and only if $X$ is 2-regular, i.e. the linearly joined squences of $r$-dimensional minimal degree varieties;
\item[(b)](next to extremal case) If $X$ fails property $\textbf{N}_{2,e}$ but satisfies $\textbf{N}_{2,e-1}$, then $X$ is arithmetically Cohen-Macaulay  and is of next to minimal degree.
\end{itemize}
\end{Coro}

\begin{proof}
For (a) we can also project $X$ to a hyperquadric $X^{(e-1)}$ similarly (In this case $X^{(e-1)}$ is reducible, i.e. a union of two $r$-linear planes and every component of $X$ degenerates into this linear subspaces). So $X^{(e-1)}$ is ACM, and from our Depth theorem \ref{depth} $X$ is also ACM, eventually 2-regular. We also have a similar result for the case of (b) by same arguments; $X$ becomes ACM and of next to minimal degree subscheme with $h^0(\mathcal I_X(2))=\frac{(e+2)(e-1)}{2}$ in this category.
\end{proof}

\begin{Remk}\label{remk_next_rigid}
For the `rigidity' of property $\textbf{N}_{2,p}$ for $p=e$, D. Eisenbud et al. proved the same theorem for the category of algebraic set more generally in \cite{EGHP1}. The (geometric) classification of 2-regularity is well-known for varieties, and for general algebraic sets it is given in \cite{EGHP2}. We reprove this rigidity (case (a)) using our inner projection method and for next to extremal case (b) we also get similar classification for the subschemes in the category $(\ast)$ (see Question \ref{Q_classify}).
\end{Remk}

\section{Examples and Open questions}\label{section_5}
It seems to be quite natural to find out a good inner projection as we move the point $q\in X$ in many aspects.  What happens to inner projections from singular points? During the discussions with P. Schenzel, we have the following example;
\begin{Ex}(Projection from a \textit{singular} point, discussion with
P. Schenzel)\label{schenzel}\\
Let us consider a singular surface $X$ in $\P^{14}$ by Segre embedding of
quadric in $\P^2$ and singular quintic rational curve in $\P^4$ (Note that `-' means ``zero" in the Betti table);

\begin{eqnarray}\label{table1}
& \begin{array}{|c|c|c|c|c|c|c|c|c|c|c|c|c|c|c|c|c|c|c|c|c|c|c|c|c|c}\hline
              & 0 & 1   & 2     & 3     & 4     & 5     &  \cdots   &  i  &  \cdots \\[1ex]
\hline
0   & 1 & -   & -     & -     & -     & -     & \cdots  &  - &  \cdots \\[1ex]
\hline
1  & - & 70 & 475 & 1605 &  3333 &  4500 & \cdots & \beta_{i,1} &  \cdots \\[1ex]
\hline
2 & - & -   & -    & 11    & 100     & 405      &   \cdots    & \beta_{i,2}  &  \cdots \\[1ex]
\hline
3 & - & -   & -    & -   & -     & -      &   \ddots    & \beta_{i,3}  &   \ddots  \\[1ex]
\hline
\end{array} &
\end{eqnarray}
\begin{center}
\textbf{Table \ref{table1}} A singular surface $X$ in $\P^{14}$ (computed by {\sc Singular})
\end{center}\bigskip 
We see that $X$ satisfies property $\textbf{N}_{2,2}$. Now consider inner projections of $X$ from (a) a smooth point and (b) any singular point of $X$ (we can't distinguish the singularities).

\begin{eqnarray}\label{table2}
(a)~~~\begin{array}{|c|c|c|c|c|c|c|c|c|c|c|c|c|c|c|c|c|c|c|c|c|c|c|c|c|c}\hline
              & 0 & 1   & 2     & 3     & 4     \\[1ex]
\hline
0   & 1 & -   & -     & -     & -      \\[1ex]
\hline
1  & - & 58 & 351 & 1035 & \cdots   \\[1ex]
\hline
2 & - & -   & 1    & 19    & \cdots       \\[1ex]
\hline
3 & - & -   & -    & -    & \ddots       \\[1ex]
\hline
\end{array}
&\phantom{X}&
(b)~~~\begin{array}{|c|c|c|c|c|c|c|c|c|c|c|c|c|c|c|c|c|c|c|c|c|c|c|c|c|c}\hline
              & 0 & 1   & 2     & 3     & 4     \\[1ex]
\hline
0   & 1 & -   & -     & -     & -      \\[1ex]
\hline
1  & - & 59 & 362 & 1089 & \cdots  \\[1ex]
\hline
2 & - & -   & -    & 10    &  \cdots    \\[1ex]
\hline
3 & - & -   & -    & -    & \ddots       \\[1ex]
\hline
\end{array}
\end{eqnarray}
\begin{center}
{\bf Table \ref{table2}}  (a) from smooth point, (b) from any singular point
\end{center}\bigskip
While $X_q$ property $\textbf{N}_{2,1}$ holds in case (a) as our Corollary \ref{Main result N_{2,p}}
says, in case (b) we still have property $\textbf{N}_{2,2}$ for $X_q$ !\\
\end{Ex}

\begin{Ex}
Let $X$ be the Grassmannian $\mathbb{G}(2,4)$ in $\P^9$, $6$-
dimensional del Pezzo variety of degree $5$ whose Betti table is

\begin{eqnarray}\label{table3}
&\begin{array}{|c|c|c|c|c|c|c|c|c|c|c|c|c|c|c|c|c|c|c|c|c|c|c|c|c|c}\hline
              & 0 & 1   & 2     & 3     \\[1ex]
\hline
0   & 1 & -   & -     & -        \\[1ex]
\hline
1  & - & 5 & 5 & -    \\[1ex]
\hline
2 & - & -   & -    & 1        \\[1ex]
\hline
 \end{array}&
\end{eqnarray}
\begin{center}
{\bf Table \ref{table3}} the Grassmannian $\mathbb{G}(2,4)$ in $\P^9$
\end{center}\bigskip
and property $\textbf{N}_{2,2}$ is satisfied. Since it is homogeneous
and covered by lines, so we can choose any (smooth) point $q$ in $X$
and a line $\ell$ through $q$ in $X$. Then the projection $X_q$ is a
complete intersection of two quadrics in $\P^8$ (property $
\textbf{N}_{2,1}$) and $q^\prime=\pi_q(\ell)$ becomes a singularity of
multiplicity $2$ in $X_q$. If we project this one more from $q^\prime
$, then the projected image becomes a quadric hypersurface in $\P^7$
still satisfying property $\textbf{N}_{2,1}$.
\end{Ex}

\begin{Qu}(Inner projection from a singular point)\label{Q_sing}
Assume that $X$ be a nondegenerate projective scheme with $\textbf{N}_{2,p}$. If $q\in X$ is singular, we could expect
that the inner projection from $q$ has more complicate aspects, but shows better behavior still satisfying $\textbf{N}_{2,p}$
in many experimental data.
What can happen to the projection from singular locus in general?
\end{Qu}

 Next, we consider Problem (a) of the introduction in general. Let $X$ be a nondegenerate  subscheme with property
 $\textbf{N}_{2,p}$. If $\ell$ meets $X$ but is not contained in $X$, then we can regard the projection $\pi_{\ell}$ as the
 composition of two simple projections from points $q_1, q_2$. Furthermore, if such $\ell$ meets $X$ at smooth two points, then
 $X_{\ell}=\overline{\pi_{\ell}(X\setminus{\ell})}$ satisfies property $\textbf{N}_{2,p-2}$ by our main Theorem.

But not the case of $\ell\subset X$ we can treat simply, because $q_2=\pi_{q_1}(\ell)$ might be a singular point
even if $\pi_{\ell} =\pi_{q_2}\circ\pi_{q_1}$. In this case, we give an interesting example showing that
the Betti numbers of  $X_{\ell}=\overline{\pi_{\ell}(X\setminus{\ell})}$ are related to the geometry of the line in $X$.

\begin{Ex}(Projection from a line \textit{inside} the variety)
Consider the Segre embedding $X=\sigma(\P^2\times\P^4)$, 6-fold of
degree $15$ in $\P^{14}$ having property $\textbf{N}_{2,3}$ whose
Betti table is

\begin{eqnarray}\label{table4}
& \begin{array}{|c|c|c|c|c|c|c|c|c|c|c|c|c|c|c|c|c|c|c|c|c|c|c|c|c|c}\hline
              & 0 & 1   & 2     & 3     & 4     & 5     & 6 & 7 & 8   \\[1ex]
\hline
0   & 1 & -   & -     & -     & -     & - & -   & -  &-  \\[1ex]
\hline
1  & - & 30 & 120 & 210 & 168 & 50 & - &- & -  \\[1ex]
\hline
2 & - &  - & -& -& 50 & 120 & 105 & 40 & 6   \\[1ex]
\hline
\end{array} &~.
\end{eqnarray}
\begin{center}
{\bf Table \ref{table4}}  Segre embedding $X=\sigma(\P^2\times\P^4)$, 6-fold of
degree $15$ in $\P^{14}$
\end{center}\bigskip
In $X$ there are two type of contained lines, so called $\ell_1$ and $\ell_2$. If we take $\ell_1$ as the line $\sigma(\{pt\}\times \ell)$ in $X$,
then the image $X_{\ell_1}$ is the intersection of two cones $\langle
\sigma(\P^1\times\P^4),\P^2\rangle$ and $\langle \P^3,\sigma(\P^2\times\P^2)\rangle$ which is a 6-fold of degree $12$
in $\P^{12}$ satisfying property $\textbf{N}_{2,2}$ with the Betti
table as in Table \ref{table5} (a).

On the other hand, in case of the line $\ell_2=\sigma(\ell\times\{pt\})$,
$X_{\ell_2}$ is a $6$-dimensional cone $\langle \{pt\},\sigma(\P^2\times\P^3)\rangle$
of degree $10$ in $\P^{12}$ and has its own Betti table as in Table \ref{table5} (b) with property $\textbf{N}_{2,3}$.
Note that the dimension of the span $\langle\cup_{q\in\ell_1}T_q X\rangle$ of tangent spaces along $\ell_1$  is $8$, but
$\dim\langle\cup_{q\in\ell_2}T_q X\rangle=10$ (i.e. the tangent spaces change
more variously along $\ell_2$ than $\ell_1$). So, it is expected that $\ell_2$
is geometrically less movable than $\ell_1$ inside $X$ and $X_{\ell_2}$ has more linear syzygies.

{\setlength\arraycolsep{2pt}
\begin{eqnarray}\label{table5}
(a)~~~\begin{tabular}{|c|c|c|c|c|c|c|c|c|c|c|c|c|c|c|c|c|c|c|c|c|c|c|c|c|c}\hline
              & 0 & 1   & 2     & 3     & 4 &5&6    \\[1ex]
\hline
0   & 1 & -   & -     & -     & -  &-&-    \\[1ex]
\hline
1  & - & 16 & 40 & 30 & 4 & - & -  \\[1ex]
\hline
2 & - & - & -& 20 & 40 & 24 & 5  \\[1ex]
\hline
\end{tabular}
&&
(b)~~~\begin{tabular}{|c|c|c|c|c|c|c|c|c|c|c|c|c|c|c|c|c|c|c|c|c|c|c|c|c|c}\hline
              & 0 & 1   & 2     & 3     & 4  &5&6  \\[1ex]
\hline
0   & 1 & -   & -     & -     & -  &-&-   \\[1ex]
\hline
1  & - & 18 & 52 & 60 & 24 & - & - \\[1ex]
\hline
2 & - & - & -& - & 10 & 12 & 3   \\[1ex]
\hline
\end{tabular}
\end{eqnarray}}
\begin{center}
{\bf Table \ref{table5}}  (a) from a line $\ell_1$ of type 1, (b) from a line $\ell_2$ of type 2
\end{center}
\end{Ex}

\begin{Qu}(Inner projection from a subvariety)\label{Q_subvar}
Let $X$ be a nondegenerate reduced scheme in $\P^N$
satisfying property $\textbf {N}_{2,p}~(p>1)$ which is not necessarily
linearly normal. Consider the inner projection from a line $\ell \subset X$. Is it true that $\overline{\pi_{\ell}(X\setminus{\ell})}$ satisfies at least $\textbf {N}_{2, p-2}$?
How does the \textit{infinitesimal} geometry of $\ell$ in $X$ effect to the syzygies of $\overline{\pi_{\ell}(X\setminus{\ell})}$?
More generally, how about the projection from a subvariety $Y$ of $X$? The projection from $Y$ is defined by the projection
from $\Lambda:=\langle Y\rangle$, the linear span of $Y$ (see \cite{BHSS}). Say $\dim \Lambda=t < p$. Does $X_\Lambda$ in $\P^{N-t-1}$
satisfy property $\textbf {N}_{2, p-t-1}$ in general as raised in the problem list (a) in the introduction?
\end{Qu}
For the sake of Question \ref{Q_subvar}, we expect to need developing the elimination mapping cone theorem and partial elimination module
theory for multivariate case and calculating on the syzygies of those partial elimination modules by Gr\"obner basis theory for graded modules.
See \cite{HK2} for basic settings and some partial results for bivariate case.\\

Finally, we have the following question as rasied in Remark \ref{remk_next_rigid}.

\begin{Qu}(Geometric characterization of some $3$-regular ACM schemes)\label{Q_classify}
We showed in Section \ref{section_3.5} that if a $r$-equidimensional, reduced and connected in codimension 1 subscheme $X$ in $\P^{r+e}$ fails property
$\textbf{N}_{2,e}$ but satisfies $\textbf{N}_{2,e-1}$, then it is a ACM, $3$-regular scheme of next to minimal degree
(i.e. $\deg(X)=\codim X + 2$) with $h^0(\mathcal I_X(2))=\frac{(e+2)(e-1)}{2}$. Further, a theorem of L.T. Hoa in \cite{HOA} gives the complete
graded Betti numbers of these schemes as follows:
$$0\to R(-e-2)\to R^{\beta_{e-2,2}}(-e)\to R^{\beta_{e-3,2}}(-e+1)\to\cdots\to R^{\beta_{0,2}}(-2)\to I_X\to 0\cdots(\ast\ast)$$
where $\beta_{i,2}=(i+1){e+1\choose i+2}-{e \choose i}$ for $0\le i \le e-2$.
Thus, just as the characterization of reduced $2$-regular projective schemes (see \cite{EGHP2}), among all the equidimensional reduced and
connected in codimension $1$ subschemes, it would be very interesting to classify or give geometric descriptions for all $3$-regular,
ACM, and next to minimal degree projective schemes whose Betti table is given by $(\ast\ast)$.
\end{Qu}


\begin{thebibliography}{PTW02}

\bibitem[AK11]{AK} J. Ahn and S. Kwak, \emph{Graded mapping cone theorem, multisecants and syzygies}, J. Algebra 331 (2011), 243--262.

\bibitem[AS10]{AS} A. Alzati and J.C. Sierra, \emph{A bound on the degree of schemes defined by quadratic equations}, to appear in `Forum Mathematicum'.

\bibitem[Bau95]{B} I. Bauer,
{\em Inner projections of algebraic surfaces: a finiteness result}, J. Reine Angew. Math. 460, 1--13 (1995).

\bibitem[BHSS00]{BHSS}
M. Beltrametti, A. Howard, M. Schneider and A. Sommese, \emph{Projections from subvarieties}, Complex Analysis and
Algebraic Geometry (T.Peternell,F-O. Schreyer, eds.), A volume in memory of Michael Schneider, (2000), 71--107.


\bibitem[BP10]{BP}
M. Brodmann and E. Park, \emph{On varieties of almost minimal degree I: Secant loci of rational normal scrolls}, J. Pure Appl. Alg. 214 (2010), 2033--2043.

\bibitem[BS07]{BS}
M. Brodmann and P. Schenzel, \emph{Arithmetic properties of
projective varieties of almost minimal degree}, J. Algebraic Geom.
16 (2007), 347--400.

\bibitem[CC01]{CC} A. Calabri and C. Ciliberto, \emph{On special projections of varieties: epitome to a theorem of Beniamino Segre},
Adv. Geom. 1 (2001), no. 1, 97--106.

\bibitem[CKK06]{CKK} Y. Choi, S. Kwak and P-L Kang, \emph{Higher linear syzygies of inner projections}, J. Algebra 305 (2006), 859--876.

\bibitem[CKP08]{CKP} Y. Choi, S. Kwak and E. Park, \emph{On syzygies of non-complete embedding of projective varieties}, Math. Zeitschrift 258, no. 2 (2008), 463--475.

\bibitem[EGHP05]{EGHP1} D. Eisenbud, M. Green, K. Hulek and S. Popescu,
{\em Restriction linear syzygies: algebra and geometry},
Compositio Math. 141 (2005), 1460--1478.

\bibitem[EGHP06]{EGHP2} D. Eisenbud, M. Green, K. Hulek and S. Popescu,
{\em Small schemes and varieties of minimal degree}, Amer. J.
Math. 128 (2006), no. 6, 1363--1389.

\bibitem[EHU06]{EHU} D. Eisenbud, C. Huneke, B.Ulrich, {\em The regularity of Tor and graded Betti numbers},
Amer. J. Math. 128 (2006), no. 3, 573--605.

\bibitem[Fuj90]{F} T. Fujita,
{\em Classification theories of polarized varieties}, Cambridge University Press, Cambridge, (1990).

\bibitem[FCV99]{FCV} H. Flenner, L. O'Carroll, W. Vogel,
{\em Joins and intersections}, Springer-Verlag, Berlin, (1999).

\bibitem[Gre84]{G2}
M. Green, \emph{Koszul cohomology and the geometry of projective varieties},  J. Differential Geom. 19 (1984), 125--171.

\bibitem[Gre98]{G}
M. Green, \emph{Generic Initial Ideals}, in Six lectures on Commutative Algebra, (Elias J., Giral J.M., Mir\'o-Roig, R.M.,
Zarzuela S., eds.), Progress in Mathematics {\bf 166}, Birkh\"auser, 1998, 119--186.

\bibitem[GL88]{GL} M. Green and R. Lazarsfeld, \emph{Some results on the syzygies of finite sets and algebraic curves},
Compositio Math. 67 (1988), no. 3, 301--314.

\bibitem[HK10]{HK2} K. Han and S. Kwak, \emph{Projections from lines : algebraic and geometric properties}, in preparation.

\bibitem[Hart77]{Ha} R. Hartshorne, \emph{Algebraic geometry}, Graduate text, Springer Verlag,

\bibitem[Hoa93]{HOA} L.T. Hoa, \emph{On minimal free resolutions of projective varieties of degree $=$ codimension + $2$}, J. Pure Appl. Algebra, 87 (1993), 241--250.

\bibitem[Katz93]{Ka} S. Katz, \emph{Arithmetically Cohen-Macaulay curves cut out by quadrics}, Computational algebraic geometry and commutative algebra (Cortona, 1991), 257--263, Sympos. Math., XXXIV, Cambridge Univ. Press, Cambridge, 1993.

\bibitem[KP05]{KP} S. Kwak and E. Park \emph{Some effects of property ${\rm N}\sb p$ on the higher normality and defining equations of nonlinearly normal varieties}, J. Reine Angew. Math. 582 (2005), 87--105.

\bibitem[OP01]{OP} G. Ottaviani and R. Paoletti, \emph{Syzygies of Veronese embeddings}, Compositio Math., 125 (2001), 31--37.

\bibitem[Park08]{P} E. Park, \emph{On secant loci and simple linear projections of some projective varieties}, preprint

\bibitem[Reid00]{R} M. Reid, \emph{Graded rings and birational geometry}, in Proc. of algebraic geometry symposium (Kinosaki, Oct 2000), K. Ohno (Ed.), 1--72.

\bibitem[Seg36]{Se} B. Segre, \emph{On the locus of points from which an algebraic variety is projected multiply}, Proc. Phys.-Math. Soc.
Japan Ser. III, 18 (1936), 425--426.

\bibitem[Som79]{So} A. Sommese, \emph{Hyperplane sections of projective surfaces I, The adjunction mapping}, Duke Math. J.
46 (1979), no. 2, 377--401.

\bibitem[Zak99]{Z} F. L. Zak, \emph{Projective invariants of quadratic embedding}, Math. Ann. 313 (1999), 507--545.

\end{thebibliography}
\end{document}